\documentclass[12pt, reqno]{amsart}
\usepackage{amsmath, amsfonts, amsthm, amssymb}
\usepackage{mdframed}

\usepackage{graphicx}
\usepackage{color}
\usepackage{xypic}
\usepackage{tikz}

\newtheorem{theorem}{Theorem}
\newtheorem{definition}[theorem]{Definition}

\newtheorem{question}[theorem]{Question}
\newtheorem{remark}[theorem]{Remark}
\newtheorem*{examples}{Examples}
\newtheorem*{example}{Example}

\newcommand{\GL}{{\rm GL}}
\newcommand{\SL}{{\rm SL}}
\newcommand{\SU}{{\rm SU}}
\newcommand{\SO}{{\rm SO}}
\newcommand{\Sp}{{\rm Sp}}
\newcommand{\Spin}{{\rm Spin}}
\newcommand{\PSL}{{\rm PSL}}
\newcommand{\Out}{{\rm Out}}
\newcommand{\Aut}{{\rm Aut}}
\newcommand{\Inn}{{\rm Inn}}

\newcommand{\Hom}{{\rm Hom}}
\newcommand{\Sym}{{\rm Sym}}

\newcommand{\Ext}{{\rm Ext}}

\newcommand{\p}{\mathsf{p}}
\newcommand{\Id}{{\rm Id}}
\newcommand{\Z}{\mathbb Z}
\newcommand{\F}{\mathbb{F}}
\newcommand{\R}{\mathbb{R}}
\newcommand{\C}{\mathbb{C}}
\newcommand{\Q}{\mathbb{Q}}
\newcommand{\G}{\mathbb{G}}

\title{Notes on Central Extensions}
\author{Dipendra Prasad }
\address{School of Mathematics, Tata Institute of Fundamental Research, Homi Bhabha Road, Colaba, Mumbai 400005, INDIA}
\email{dprasad@math.tifr.res.in}

\begin{document}
\maketitle

\begin{abstract}
These are the notes for some lectures given by this author at Harish-Chandra Research Institute, Allahabad in March 2014 for a workshop on Schur multipliers. The lectures aimed at giving an overview of the subject with emphasis on groups of Lie type over finite, real and $p$-adic fields. The author thanks Prof. Pooja Singla for the first draft of these notes and Shiv Prakash Patel for the second draft of these notes.
\end{abstract}
\section{Introduction}

\begin{definition}

Let $G$ be a group and $A$ an abelian group. A group $E$ is called a central extension of $G$ by $A$ if there is a short exact sequence
of groups,
\begin{equation}
1 \rightarrow A \xrightarrow{i} E \xrightarrow{\p} G \rightarrow 1  
\end{equation}
such that image of $i$ is contained in the center of $E$.
\end{definition}
Isomorphism classes of central extensions of $G$ by $A$ are parametrized by $H^{2}(G, A)$, where $A$ is considered to be a trivial $G$-module. Indeed, if $s : G \rightarrow E$ is a section of $\p$ then $\beta : G \times G \rightarrow A$ given by $\beta(g_1, g_2) := s(g_1)s(g_2)s(g_1 g_2)^{-1}$ defines a 2-cocycle on $G$ with values in $A$. On the other hand, if $\beta : G \times G \rightarrow A$ is a 2-cocycle on $G$ with values in $A$, then the binary operation on $E=G \times A$ defined by $(g_{1}, a_{1}) (g_{2}, a_{2}) := (g_{1} g_{2}, a_{1}a_{2} \beta(g_{1}, g_{2}))$ makes $E$ into a group, which is a central extension of $G$ by $A$.

A closely related group which comes up in the study of central extensions is the {\it Schur multiplier} of  a group $G$ defined to be $H_{2}(G, \Z)$. If $G$ is perfect, by the {\it universal coefficient theorem} recalled below, $H^{2}(G, \Q/\Z) \cong \Hom(H_{2}(G, \Z), \Q/\Z)$, i.e. Pontryagin dual of $H_{2}(G, \Z)$ classifies central extensions of $G$ by $\Q/\Z$ (if $G$ is perfect).
\begin{theorem}[Universal Coefficient theorem] 
Let $G$ be a group and $A$ an abelian group, considered as a trivial $G$-module. Then we have a short exact sequence of abelian groups as follows:
\begin{displaymath}
\xymatrix@1{ 0 \ar[r] & \Ext^{1}(H_{1}(G, \Z), A) \ar[r] & H^{2}(G, A) \ar[r] & \Hom( H_{2}(G, \Z), A)  \ar[r] & 0. }
\end{displaymath}
This sequence is split  (though not in a natural way).
\end{theorem}
Let us consider the following two extreme cases of $G$. 
\begin{enumerate}
\item $G^{ab}  :=G/[G, G] = H_{1}(G, \Z)= 1$.
\item $G = G^{ab}$.
\end{enumerate}
Let us consider the first case, i.e. $G^{ab} = 1 $, in which case $H_{1}(G, \Z) = 1$. In this case, the universal coefficient theorem reduces to,
\begin{equation}
  H^{2}(G, A) \cong \Hom( H_{2}(G, \Z), A).
\end{equation}
In particular, there exists a central extension $\tilde{G}$ of $G$ by $A=H_{2}(G, \Z)$ corresponding to the identity, $\Id \in \Hom( H_{2}(G, \Z), H_{2}(G, \Z))$:
\begin{equation}
1 \rightarrow H_{2}(G, \Z) \rightarrow \tilde{G} \rightarrow G \rightarrow 1.  
\end{equation}
 This central extension $\tilde{G}$ of $G$ by $H_{2}(G, \Z)$ is universal in the sense that any central extension $E$ of $G$ by  $A$ is given by a push-out diagram for a group homomorphism $\varphi : H_{2}(G, \Z) \rightarrow A$ as follows:
 \begin{displaymath}
 \xymatrix{
 0 \ar[r] & H_{2}(G, \Z) \ar[r] \ar[d]_{\varphi} & \tilde{G} \ar[r] \ar[d] & G \ar[r] \ar[d] & 0 \\
 0 \ar[r] & A \ar[r] & E \ar[r] & G \ar[r] & 0,
 }
\end{displaymath}
where $E = [\tilde{G} \times A]/\Delta {H_{2}(G, \Z)} $. \\

In the second extreme case, we have $G = G^{ab}$, i.e. $G$ is abelian. For an abelian group $G$, we have 
\begin{equation} 
H_{2}(G, \Z) \cong \Lambda^{2}G := \dfrac{G \otimes G}{\{ g \otimes g : g \in G \}}.
\end{equation}

The universal coefficient theorem for an abelian group $G$ gives a split exact sequence as follows:
 \begin{displaymath} \tag{*} \label{*} 
\xymatrix@1{ 0 \ar[r] & \Ext^{1}(G, A) \ar[r] & H^{2}(G, A) \ar[r]  & \Hom( \Lambda^{2}G, A)  \ar[r] & 0. } 
 \end{displaymath}
 This exact sequence can be nicely interpreted in terms of extensions. Recall that $H^{2}(G, A)$ classifies central extensions $E$ of $G$ by $A$. Among these central extensions $E$, those which are abelian correspond to the subgroup $\Ext^{1}(G, A)$ 
of $H^{2}(G, A)$. The map from $H^{2}(G, A)$ to $\Hom(\Lambda^{2}G, A)$ is given by taking arbitrary lifts $g_{1}, g_{2}$ elements of $G$ to $E$, and taking their commutator in $E$ which is an element of $A$. This clearly gives a homomorphism from $\Lambda^{2}G$ to $A$. 
In particular, if $A$ is a divisible abelian group such as $A = \Q/\Z$, then since $\Ext^{1}(G, A) = 0$ (for $G$ abelian),  
\[
H^{2}(G, A) \cong \Hom(\Lambda^{2}G, A).
\]

The exact sequence (\ref{*}) is known to be split.  We show by 
an example that the exact sequence (\ref{*}) is not canonically split.
 For this, take $G=V=\Z/2+\Z/2$, and $A=\Z/2$. It is known that
$H^*(\Z/2,\Z/2)$ 
is the polynomial algebra on $H^1(\Z/2,\Z/2) \cong \Z/2$, and therefore by the Kunneth theorem, $H^*(\Z/2+ \Z/2,\Z/2) \cong \Sym^{*}[V^\vee] \cong  \Z/2[X,Y]$.
In particular,   $H^2(\Z/2+ \Z/2,\Z/2) \cong  \Sym^2[V^\vee]$.
Further, 
$\Ext^1(V,\Z/2) \cong V^\vee$. The exact sequence:

$$\xymatrix@1{ 0 \ar[r] & \Ext^{1}(G, A) \ar[r] & H^{2}(G, A) \ar[r]  & \Hom( \Lambda^{2}G, A)  \ar[r] & 0, } $$
 becomes,
$$\xymatrix@1{ 0 \ar[r] & V^\vee \ar[r] & \Sym^2[V^\vee] \ar[r]  & \Z/2  \ar[r] & 0, } $$
where the map from $V^\vee$ to $\Sym^2[V^\vee]$ is given by $v\rightarrow v\otimes v$. It is known that for $\Aut(G) = \Aut(V) = \GL_2(\Z/2)$, the above 
is a nonsplit exact sequence of $\Aut(V)$-modules.

\section{The Dual point of view}
Recall that a central extension of a group $G$ by an abelian group $A$ is another group $E$ such that $A$ is contained in the center of $E$ and $E/A \cong G$, i.e. we have an exact sequence of groups
\[
1 \rightarrow A \rightarrow E \rightarrow G \rightarrow 1.
\] 
A common theme in categorical mathematics is that many notions in an abstract category remain meaningful by reversing arrows! Reversing arrows in the above exact sequence, we get:
\[
1 \rightarrow G \rightarrow E \rightarrow A \rightarrow 1,
\]
i.e., now $G$ sits as a normal subgroup inside a group $E$ with quotient $A$, which we don't necessarily assume to be abelian, 
and then change notation from $A$ to $Q$.

We ask the following question from a dual point of view to the central extension.
\begin{question}
For a group $G$ and another group $Q$, what are the way in which $G$ sits in a group $E$ as a normal subgroup with quotient $E/G \cong Q$? In other words, what are the isomorphism classes of extension of groups $E$ which give rise to a short exact sequence of groups as follows:
\[
1 \rightarrow G \rightarrow E \rightarrow Q \rightarrow 1.
\]
\end{question}
Note that given a normal subgroup $G$ of a group $E$ with quotient $Q$, there exists a natural homomorphism $\phi : Q \rightarrow \Out(G)$, where $\Out(G)$ is the group of outer automorphism of $G$ which is defined as $\Aut(G)/\Inn(G)$, where $\Aut(G)$ is the group of all automorphisms of $G$ and $\Inn(G)$ 
is the normal subgroup of the group of inner automorphisms of $G$. The mapping $\phi : Q \rightarrow \Out(G)$ is defined by choosing an arbitrary lift $\tilde{a}$ of $a \in Q$ in $E$, and using the automorphism of $G$ given by conjugation by $\tilde{a}$, the automorphism of $G$ considered as an element of $\Out(G)$ being independent of the choice of the lift $\tilde{a}$ of $a$.\\

The group  $\Out(G)$ is an important invariant of a group $G$, which has been much studied in all branches of mathematics where groups play a role. \\

Let $C$ be the center of $G$. An element  $\xi \in H^{2}(Q, C)$ corresponds to a central extension
\[
0 \rightarrow C \rightarrow E \rightarrow Q \rightarrow 0.
\]
This gives rise to a push-out diagram:
\begin{displaymath}
\xymatrix{
0 \ar[r] & C \ar[r] \ar[d] & E \ar[r] \ar[d] & Q \ar[r] \ar[d] & 0 \\
0 \ar[r] & G \ar[r] & \tilde{E} \ar[r] & Q \ar[r] & 0, }
\end{displaymath}
where $\tilde{E} := [E \times G]/\Delta(C)$. Thus we have a natural map 
\begin{equation}
H^{2}(Q, C) \rightarrow \Ext^{1}(Q, G).
\end{equation}
As described above, $ 0 \rightarrow G \rightarrow \tilde{E} \xrightarrow{\p} Q \rightarrow 0$ gives rise to a natural map,
\begin{equation} \label{ext 2 hom}
\Ext^{1}(Q, G) \rightarrow \Hom(Q, \Out(G)).
\end{equation} 

Given a homomorphism $\varphi : Q \rightarrow \Out(G)$, we have a pull-back diagram of the natural exact sequence
\[
1 \rightarrow G/C \rightarrow \Aut(G) \rightarrow \Out(G) \rightarrow 1,
\] given by
\begin{displaymath}
\xymatrix{
1 \ar[r] & G/C \ar[r] \ar@{=}[d] & E=\Aut(G) \times_{\Out(G)} Q \ar[r] \ar[d]^{\varphi} & Q \ar[r] \ar[d] & 1 \\
1 \ar[r] & G/C \ar[r] & \Aut(G) \ar[r] & \Out(G) \ar[r] & 1.
}
\end{displaymath}
This gives us a morphism of sets $\Hom(Q, \Out(G)) \rightarrow \Ext^{1}(Q, G/C)$.

Associated to the short exact sequence 
\[
1 \rightarrow C \rightarrow G \rightarrow G/C \rightarrow 1,
\]
we have an exact sequence of pointed sets 
\begin{displaymath}
\xymatrix{ 0 \ar[r] & \Ext^{1}(Q, C) \ar[r] \ar@{=}[d] & \Ext^{1}(Q, G) \ar[r] \ar[rd] & \Ext^{1}(Q, G/C) \ar[r] & \cdots \\
& H^{2}(Q, C) & & \Hom(Q, \Out(G)) \ar[u] & }
\end{displaymath}
giving rise to
\[
0 \rightarrow H^{2}(Q,C) \rightarrow \Ext^{1}(Q,G) \rightarrow \Hom(Q, \Out(G)).
\]
The following theorem summarizes the above discussion.
\begin{theorem}
Let $C$ be the center of a group $G$. Then for any group $Q$, we have an exact sequence of pointed sets:
\[
0 \rightarrow H^{2}(Q, C) \rightarrow \Ext^{1}(Q, G) \rightarrow \Hom(Q, \Out(G)).
\] 
making $\Ext^1(Q, G)$ into a principal homogeneous space with fibres abelian groups $H^{2}(Q, C)$ and base which is the subset of 
$\Hom(Q, \Out(G))$ consisting of those homomorphisms from $Q$ to $\Out(G)$  which is realized by an extension of $Q$ by $G$.

If the center of $G$ is trivial, then
\[
\Ext^{1}(Q, G) \cong \Hom(Q, \Out(G)).
\]
\end{theorem}
\begin{proof} We only make some remarks on the last assertion in the theorem. If $C= \{ e \}$,
then the morphism of sets $\Hom(Q, \Out(G)) \rightarrow \Ext^{1}(Q, G/C)$,
 gives rise to a morphism of sets 
\[
\Hom(Q, \Out(G)) \rightarrow \Ext^{1}(Q, G).
\]
This can be easily seen to be inverse of the morphism of sets in (\ref{ext 2 hom}) and therefore if $C = \{ e \}$, we have an isomorphism of sets
\[
\Ext^{1}(Q, G) \cong \Hom(Q, \Out(G)).
\]
\end{proof}

\begin{example} \normalfont
As an example, we consider possible extensions of $\Z/2\Z$ by $\GL_{n}(\C)$
\[
1 \rightarrow \GL_{n}(\C) \rightarrow E \rightarrow \{\pm1\} \rightarrow 1,
\]
such that $-1$ acts on $\GL_{n}(\C)$ via the outer automorphism given by $g \mapsto ^{t}g^{-1}$. We note that the outer automorphism group of $\GL_{n}(\C)$ is isomorphic to $\Z/2$ generated by $g \mapsto ^{t}g^{-1}$. By the above theorem, the number of extensions corresponding to this outer automorphism of $\GL_{n}(\C)$ is exactly the number of elements in $H^{2}(\{\pm1\}, Z(\GL_{n}(\C))) = H^{2}(\{\pm1\}, \C^{\times})$, where $\{\pm1\}$ acts on $\C^{\times}$ by $x \mapsto x^{-1}$. Since the 
cohomology of a finite cyclic group is periodic with period 2, the 2nd cohomology reduces to the 0th (Tate) cohomology, which is easily seen to be $\Z/2$ in this case. This means that there are exactly two extensions of $\Z/2$ by $\GL_n(\C)$, one of which is the trivial one, i.e., a semi-direct product, and the other one given by 
generators and relations  as $\{\GL_n(\C); j\}$ with $j^2=-1, jgj^{-1} = j_0^tg^{-1}j_0^{-1}$ with $j_0$ the anti-diagonal matrix with entries $(1,-1,1,-1,\cdots)$.
\end{example}

\section{Examples}
\begin{examples}[Examples of central extensions of certain abelian groups:] \normalfont
\noindent
\begin{enumerate}
\item Central extensions of $\Z/2\Z$ by $\Z/2\Z$: There are two extensions in this case. One of them is the trivial extension, i.e., the 
direct product, and  the other one is the non-trivial extension, 
\[
0 \rightarrow \Z/2\Z \rightarrow \Z/4\Z \rightarrow \Z/2\Z \rightarrow 0.
\]

\item Let $\F_{q}$ be the finite field with $q=p^n$-elements. There is a non-trivial central extension $H(\F_{q}^{2n})$ of $\F_{q}^{2n}$ by $\F_{q}$ for all positive integer $n$, called the Heisenberg group, i.e.  
\[
0 \rightarrow \F_{q} \rightarrow H(\F_{q}^{2n}) \rightarrow \F_{q}^{2n} \rightarrow 0.
\]The simplest realization of $H(\F_q^{2n})$ is the group of $(n+1) \times (n+1)$ upper triangular unipotent matrices with 
only non-diagonal nonzero entries  (from $\F_q$) in the first row and last column. 
\end{enumerate}
\end{examples}

\begin{example}[Central extension of  $\{ \pm 1 \}^{n-1}$]
\normalfont
Let ${\rm O}_{n}(\R)$ 
be defined as isometry group of the quadratic form given by $q(x) = x_{1}^{2} + \cdots + x_{n}^{2}$, and let $\SO_{n}(\R)= {\rm O}_{n}(\R) \cap \SL_{n}(\R)$.
Consider the $\{ \pm 1 \}^{n} \hookrightarrow {\rm O}_{n}(\R)$ as the group of diagonal matrices with entries $+1$ or $-1$. Identify $\{ \pm 1 \}^{n-1}$ to 
be the subgroup of $\{ \pm 1\}^{n}$ consisting of those elements of $\{ \pm 1 \}^{n}$ with even number of $-1$'s.  Define $F_{n-1}$ to be the 2-fold cover of $\{ \pm 1 \}^{n-1}$ which is obtained from the pull back of $\{ \pm 1 \}^{n-1} \hookrightarrow \SO_{n}(\R)$ through the spin cover 
$\Spin_{n}(\R) \rightarrow \SO_{n}(\R)$, i.e.
\begin{displaymath}
\xymatrix{
F_{n-1} \ar[r] \ar[d] & \Spin_{n}(\R) \ar[d] \\
\{ \pm 1 \}^{n-1} \ar[r] & \SO_{n}(\R). }
\end{displaymath}
Then $F_{n-1}$ is a non-trivial 2-fold cover of $\{ \pm 1 \}^{n-1}$. Let us describe the group $F_{n-1}$ more explicitly (via the usual Clifford 
algebra construction which we hide). \\
We will describe $F_{n-1}$ as a subgroup of another group $E_{n}$ defined below. As a set $E_{n}$ is 
\[
E_{n} = \{ \epsilon e_{A} \mid \epsilon \in \{ \pm 1 \}, A \subset \{ 1,\cdots, n \} \},
\]
where $\{ \pm 1 \}$ lies in the center of $E_{n}$ and
\[
F_{n-1} = \{ \epsilon e_{A} \mid \#(A) =\text{even } \}.
\]
For $A = \{ e_{i_1}, \cdots, e_{i_j} \}$ with $i_1 < \cdots < i_j$, we write $e_{A} := e_{i_1} \cdots e_{i_j}$ with convention that $e_{\emptyset} = 1$. The elements $e_{i}$'s satisfy the relation $e_{i}e_{j} = - e_{j}e_{i}$ for $i \neq j$ and $e_{i}^{2} = 1$. For $\epsilon_{1}, \epsilon_{2} \in \{ \pm 1 \}$ and $A, B \subset \{1,\cdots,n \}$, define  the multiplication as $$(\epsilon_1 e_A)(\epsilon_2 e_B) = \epsilon_{1} \epsilon_{2} \epsilon(A, B) e_{(A \cup B) - (A \cap B)},$$ 
where $\epsilon(A, B)$ is determined by the relations relations among $e_{i}$'s given above. 
\begin{theorem}
\noindent
\begin{enumerate}
\item[(a)] $E_{n}$ and $F_{n-1}$ define a non-trivial central extension of $\{ \pm 1 \}^{n}$ and $\{ \pm 1 \}^{n-1}$ by $\{ \pm 1 \}$ respectively. $F_{n-1}$ is isomorphic to the one obtained from 2-fold cover $\Spin(n) \rightarrow \SO(n)$.
\item[(b)] $[E_{n}, E_{n}] = \{ \pm 1 \}$.
\item[(c)] The center of $E_{n}$ is $\{ \pm 1 \}$ if $n$ is even, and $\{ \pm 1, \pm e_{1} \cdots e_{n} \}$ if $n$ is odd, whereas 
the center of $F_{n-1}$ is $\{ \pm 1 \}$ if $n$ is odd, and $\{ \pm 1, \pm e_{1} \cdots e_{n} \}$ if $n$ is even. 
\end{enumerate}
\end{theorem}
\end{example}

\begin{remark}The group $E_n$ for $n$ even, and $F_{n-1}$ for $n$ odd is what's called an extra special 2 group. They have a unique 
irreducible representation of dimension $> 1$, which is equal to $2^{[n/2]}$ where $[n/2]$ refers to the integral part of $n/2$.
\end{remark}

\begin{example}[Central extensions of alternating groups] 
\normalfont
Let $S_{n}$ be the symmetric group on a finite set of $n$ elements. Let $A_{n} \subset S_{n}$ be the subgroup of even permutations. It is known that $H_{2}(A_{n}, \Z) = \Z/ 2\Z$ if $n=5$ or $n>7$; $H_{2}(A_{6}, \Z) \cong H_{2}(A_{7}, Z) \cong \Z/6\Z$. In particular, $A_{n}$ for $n \geq 5$ have a unique 2-fold cover $\tilde{A}_{n}$. We construct $\tilde{A}_{n}$ using the spin cover of $\SO(n)$ below.\\
Let $\{e_{1}, \cdots, e_{n}\}$ be an orthonormal basis of an $n$-dimensional quadratic space over $\R$.  We identify $S_{n}$ with the group of matrices, which permute the basis vectors $\{ e_{1}, \cdots, e_{n} \}$. This gives rise to an embedding $S_{n} \hookrightarrow O_{n}(\R)$ and hence $A_{n} \hookrightarrow \SO_{n}(\R)$. Define $\tilde{A}_{n}$ to be the 2-fold cover of $A_{n}$ which is obtained from the pull-back of the $A_{n} \rightarrow \SO_{n}(\R)$ and $\Spin_{n}(\R) \rightarrow \SO_{n}(\R)$, i.e. 
\begin{displaymath}
\xymatrix{
\tilde{A}_{n} \ar[r] \ar[d] & \Spin_{n}(\R) \ar[d] \\
A_{n} \ar[r] & \SO_{n}(\R). }
\end{displaymath}
Then $\tilde{A}_{n}$ is a non-trivial 2-fold cover of $A_{n}$.

Let us directly describe the 2-fold cover of $A_{n}$, which arises in the above fashion. We will define a two fold cover of $S_{n}$ such that the 2-fold cover of $A_{n}$ is obtained from that of restriction to $A_{n}$. The group $S_{n}$ has a presentation on $n-1$ generators, say $t_{1}, \cdots, t_{n-1}$ with the following relations:
\begin{enumerate}
\item[(a)] $t_{i}^{2} = 1$ for $1 \leq i \leq n-1$.
\item[(b)] $t_{i+1}t_{i}t_{i+1} = t_{i} t_{i+1} t_{i}$ for $1 \leq i \leq n-2$.
\item[(c)] $t_{j} t_{i} = t_{i} t_{j}$ for $|i - j | > 1$.
\end{enumerate}
We use these relations to describe a two fold cover $\tilde{S}_{n}$ of $S_{n}$. The group $\tilde{S}_{n}$ has generators $z, t_{1}, \cdots, t_{n}$ with the following relation:
\begin{enumerate}
\item[(a)] $z^{2} = 1$.
\item[(b)] $t_{i} t_{i} = z$ for $1 \leq i \leq n-1$.
\item[(c)] $t_{i+1}t_{i}t_{i+1} = t_{i} t_{i+1} t_{i}$ for $1 \leq i \leq n-2$.
\item[(d)] $t_{j} t_{i} = t_{i} t_{j}$ for $|i - j | > 1$.
\end{enumerate}
\begin{theorem}
Let $\tilde{A}_{n}$ be the 2-fold cover of $A_{n}$ which is the restriction of the above defined cover $\tilde{S}_{n}$ of $S_{n}$.
\begin{enumerate}
\item The $\tilde{A}_{n}$ is a non-trivial 2-fold cover of $A_{n}$ if and only if $n \geq 4$.
\item For $n \geq 4$, up to an isomorphism, $\tilde{A}_{n}$ is the only non-trivial 2-fold cover of $A_{n}$, which is isomorphic to the 2-fold cover of $A_{n}$ obtained from the pull back of 2-fold cover $\Spin_{n}(\R) \rightarrow \SO_{n}(\R)$.
\end{enumerate}
\end{theorem}
\end{example}

\section{Real groups}

For $G$ a  semisimple real group, we randomly pick a few examples.
\begin{enumerate}
\item $\SU(n)$ is simply connected.
\item Simply connected cover of $\SO(n)$ is $\Spin(n)$, which is a 2-fold cover of $\SO(n)$, i.e. we have
\[
0 \rightarrow \Z/2\Z \rightarrow \Spin(n) \rightarrow \SO(n) \rightarrow 1.
\]
\item $\pi_{1}(\SL_{n}(\R)) = \Z/2$ if $n>2$, and equals $\Z$ for $n=2$. 

\item 
$\pi_1(\Spin(p, q)) = 1, \Z, \Z/2$ depending on whether $\min(p,q) \leq 1$, $\min(p,q)=2$, $\min(p,q)>2$.

\end{enumerate}

\begin{remark}Information on the (topological) fundamental group of a semi-simple simply connected real algebraic group $G(\R)$ seems not to be 
clearly spelled out in the literature.  It is known that the only options for $\pi_1(G(\R)) = \pi_1(K)$ for $K$ a maximal compact subgroup of $G(\R)$,
 are $1, \Z/2,\Z$, and the only cases when $\pi_1(K) = \Z$ are the Hermitian symmetric cases. 
The only cases in which the fundamental group is trivial
is for $\SU^*(2n), \Sp(p,q), \Spin(n,1)$, the non-split inner form of $E_6$, and the rank 1 form of $F_4$. 
It is known that real forms $G'$ of a group $G$ are in 
bijective correspondence with conjugacy classes of involutive automorphisms $\sigma$ on the compact real form $K$ of $G$, such that the 
maximal compact subgroup of $G'$ is $K'=K^\sigma$. Since $K'$ has the same fundamental group as $G'$, the question on $\pi_1(G')$ amounts to a 
question in algebraic groups: for an involution on a simple simply connected algebraic group, when is the fixed points semi-simple, 
and when it is semi-simple and simply connected.
\end{remark}

\section{Two questions}
We pose two questions here which we are not sure are already answered in existing literature!
\begin{question}
Let $G$ be any finite or compact connected Lie group. Is any two fold cover of $G$ obtained from the pull-back of $\Spin(n)$, the two fold spin cover of $\SO(n)$, through a map $\varphi : G \rightarrow \SO(n)$?
\end{question}
\begin{question}
Is $H^{*}(G, \Z/2\Z)$ generated as an algebra by $\omega_{i}(\rho)$, as $\rho : G \rightarrow {\rm O}(n)$ varies over {\it all} orthogonal representations of $G$ and $\omega_{i}(\rho)$'s are Stiefel-Whitney classes of orthogonal representations $\rho$?
\end{question} 
\begin{remark} The question seems analogous to the  Hodge conjecture in Algebraic Geometry 
which is about generators of certain cohomologies by Chern classes of vector bundles. 
\end{remark}
\section{Relation of central extensions to $K_{2}$}
\begin{theorem}[Steinberg]
If $n \geq 3$,
then $H_{2}(\SL_{n}(F), \Z) \cong K_{2}(F)$ provided that we exclude $n=3$ and $|F| = 2$ or $4$, and $n=4$ and $|F|=2$.
\end{theorem}
\begin{remark} \normalfont
For $a, b \in F^{\times}$ and $n \geq 3$, let
\begin{displaymath}
r_{a} = \left( \begin{matrix} a & & & & \\ & a^{-1} & & & \\ & & 1 & & \\ & & & \ddots & \\ & & & & 1\end{matrix} \right),
\tau_{b} = \left( \begin{matrix} b & & & & \\ & 1 & & & \\ & & b^{-1} & & \\ & & & \ddots & \\ & & & & 1\end{matrix} \right) \in \SL_{n}(F).
\end{displaymath}
Let $1 \rightarrow C \rightarrow E \xrightarrow{\p} \SL_{n}(E) \rightarrow 1$ be a central extension of $\SL_{n}(F)$. 
Note that $[\tilde{r}_{a}, \tilde{\tau}_{b}]$ is a well defined element in $E$, where $\tilde{r}_{a}$ and $\tilde{\tau}_{b}$ are arbitrary element in the inverse images $\p^{-1}r_{a}$ and $\p^{-1}\tau_{b}$ respectively. 
Since $r_{a}$ and $\tau_{b}$ commute in $\SL_{n}(F)$, $[\tilde{r}_{a}, \tilde{\tau}_{b}] \in C$. 
This defines a map $(a, b) \mapsto \{a, b\}:=[\tilde{r}_{a}, \tilde{\tau}_{b}] \in C$ from $F^{\times} \times F^{\times} \rightarrow C$. 
Moreover,
\[
 \{a, b_{1}b_{2} \} = [\tilde{r}_{a}, \tilde{\tau}_{b_1 b_2}]= [\tilde{r}_a, \tilde{\tau}_{b_1} \tilde{\tau}_{b_2}] = [\tilde{r}_{a}, \tilde{\tau}_{b_1}] \cdot [\tilde{r}_{a}, \tilde{\tau}_{b_2}] \cdot [\tilde{\tau}_{b_1}, [ \tilde{\tau}_{b_2}, \tilde{r}_{a} ] ]^{-1} = \{a, b_1\} \{a, b_2\}.
\]
Similarly $\{a_1 a_2, b\} = \{a_1, b\} \{a_2, b\}$.
Therefore this map is bi-multiplicative in both the coordinates. 
An important point to note is that instead of coroots $r_a$, $\tau_b$ we could have taken any other two distinct coroots for defining the element $\{a,b\}$. This follows because the Weyl group operates transitively on the set of distinct coroots with nonzero inner product. (If the coroots are orthogonal to each other, then they belong to distinct commuting $\SL_2(F)$, and hence their commutator in any central extension is trivial.) 
Furthermore, it can be verified that $\tilde{r}_a \tilde{r}_b = \{a,b\} \tilde{r}_{ab}$ which allows one to prove the last two of the following 
identities (the first being trivial); for all this ---which although is no more than matrix manipulation, is quite tricky, and due to Steinberg--- 
see the book of Milnor \cite{Mil71}.
\begin{enumerate}
 \item $\{a,b\} \{b,a\} =1$.
 \item $\{a, -a\} =1$.
 \item $\{a, 1-a\} =1$.
\end{enumerate}
Thus the map $(a,b) \mapsto \{a,b\}$ factors through 
\[
K_{2}(F) := \dfrac{F^{\times} \otimes F^{\times}}{ \{ a \otimes (1-a) \mid a(1-a) \neq 0 \} }.
\]
Since $\SL_{n}(F)$, $n \geq 3$ is easily seen to be a perfect group, $\SL_{n}(F)$ has a universal central extension with $H_{2}(\SL_{n}(F), \Z)$, as the center of the universal central extension. The above analysis with the universal central extension of $\SL_{n}(F)$ by $H_{2}(\SL_{n}(F), \Z)$ gives a map $K_{2}(F) \rightarrow H_{2}(\SL_{n}(F), \Z)$ which by a theorem of Steinberg is an isomorphism.
\end{remark}

The following more general theorem is due to Matsumoto.

\begin{theorem} 
Let $G$ be a simple, simply connected split algebraic group over an infinite field $F$. Then
\[
H_{2}(G(F), \Z) \cong K_{2}(F)
\]
except for  groups of type $C_{n}, n\geq 1$, when $H_{2}(G(F), \Z)$ has $K_{2}(F)$ as a quotient, and in fact $K_2(F)$ is the maximal quotient of 
$H_2(G(F), \Z)$ on which ${\rm Aut}(G)(F)$ acts trivially.
\end{theorem} 

\begin{remark}
It is easy to see that $K_{2}(\F_{q}) = \{ 1 \}$ for a finite field $\F_{q}$. Also, it is also known that
 simple simply connected algebraic groups $G$, $G(\F_{q})$ have no non-trivial central extensions except in a small number of cases that we enumerate later. 
But the author of these notes has not seen any uniform theorem proving this such as the very precise theorem 12 above due to Steinberg 
(the problem is for small fields such as $\F_2, \F_3, $ and $\F_4$).
\end{remark}

\begin{remark}There seems no such precise theorem for quasi-split groups over general fields. Deodhar has defined in \cite{Deo} what he calls a Moore group
which depends only on the field of which  $H_2(G(F),\Z)$ is a quotient of.   

\end{remark}


\begin{theorem}
If $G$ is simply connected simple algebraic group defined over $\F_{q}$, then
\begin{enumerate}
\item $G^{ab} = \{ e \}$, i.e. $G$ is perfect, except the following: 
\[ \SL_{2}(\F_{2}) \cong S_{3}, \, \SL_{2}(\F_{3}) \cong \tilde{A}_{4}, \, \Sp_{4}(\F_{2}), \,  G_{2}(\F_{2}), \, \SU_{3}(\F_{2}). \]

\item If $G=G^{ab}$, then $G$ is its own universal central extension, except if $G$ is one of the following:

\[ 
\SL_{2}(\F_{4}) \cong \PSL_{2}(\F_{5}), \, \SL_{2}(\F_{9}), \, \SL_{3}(\F_{2}) \cong \PSL_{2}(\F_{7}), \, \SL_{3}(\F_{4}), 
\]
\[ 
\SL_{4}(\F_{2}) \cong A_8, \, \Spin_{7}(\F_{2}) \cong \Sp_{6}(\F_{2}), \, \Sp_{4}(\F_{2}) \cong S_6, \, \Spin_{7}(\F_{3}), \, \Spin_{8}(\F_{2}),  
\]
\[ 
F_{4}(\F_{2}), \ G_{2}(\F_{3}), \, G_{2}(\F_{4}), \, \SU_{4}(\F_{2}), \, \SU_{4}(\F_{3}), \, \SU_{6}(\F_{2}), \, ^{2}E_{6}(\F_{2}). 
\]

\end{enumerate}

\end{theorem}

\begin{remark} \normalfont
Let $Z$ be the center of a connected algebraic group $G$ defined over $\F_{q}$ which we assume is a finite (algebraic) group. It is curious to observe that the central extension
\[
1 \rightarrow Z(\F_{q}) \rightarrow G(\F_{q}) \rightarrow G(\F_{q})/Z(\F_{q}) \rightarrow 1
\]
has a ``dual extension", given by the following part of the long exact sequence associated to the exact sequence $1 \rightarrow Z \rightarrow G \rightarrow G/Z \rightarrow 1$
of algebraic groups:
\[
1 \rightarrow Z(\F_{q}) \rightarrow G(\F_{q}) \rightarrow (G/Z)(\F_{q}) \rightarrow H^{1}({\rm Gal}(\overline{\F}_{q} / \F_{q}), Z) \rightarrow H^{1}({\rm Gal}(\overline{\F}_{q} / \F_{q}), G) \rightarrow \cdots,
\]
but $H^{1}({\rm Gal}(\overline{\F}_{q} / \F_{q}), G) = \{ 1 \}$ by Lang's theorem. Thus we have 
\[
1 \rightarrow Z(\F_{q}) \rightarrow G(\F_{q}) \rightarrow (G/Z)(\F_{q}) \rightarrow H^{1}({\rm Gal}(\overline{\F}_{q} / \F_{q}), Z) \rightarrow 1.
\]
It is well known that for ${\rm Gal}(\overline{\F}_{q} / \F_{q}) = \hat{\Z}$, and for a module $A$ of $\hat{\Z}$ with $A^\vee=Hom(A,\Q/\Z)$ with natural $\hat{\Z}$ structure, 
there is a perfect pairing:

$$H^1(\hat{\Z}, A) \times H^0(\hat{\Z}, A^\vee)\rightarrow H^1(\hat{\Z}, \Q/\Z) = \Q/\Z,$$
 and hence
\[
H^{1}({\rm Gal}(\overline{\F}_{q} / \F_{q}), Z) \cong  Z(\F_{q}). 
\] 
Therefore the above part of the long exact can be written as the following short exact sequence:
\[
1 \rightarrow G(\F_{q})/Z(\F_{q}) \rightarrow (G/Z)(\F_{q}) \rightarrow Z(\F_{q}) \rightarrow 1.
\]
\end{remark}

\section{Central extension of Algebraic groups}

The following basic theorem is due to C. Moore, Matsumoto, Deodhar, G. Prasad, Raghunathan, and Rapinchuk.

\begin{theorem} Let $G$ be an absolutely simple, simply connected algebraic group which is isotropic over $k$, a  
non-Archimedean local field, with $\mu(k)$ the cyclic group of roots of unity in $k$, or $k=\R$, and $G$ split but not of type $C_n, n \geq 1$.
Then there exists a natural isomorphism $H^2(G(k), \Q/\Z) 
\cong \Hom(\mu(k), \Q/\Z).$
\end{theorem}

Given this theorem, two important questions immediately come to mind (which is also the way the theorem is proved):

\begin{enumerate}
\item What is the functorial nature of the group $H^2(G(k), \Q/\Z) $ as the field $k$ varies?

\item What is the functorial nature of the group $H^2(G(k), \Q/\Z) $ as the group $G$  varies?

\end{enumerate}

The following theorem due to C. Moore, Deodhar, G. Prasad, Raghunathan, and Rapinchuk is useful to answer such questions.

\begin{theorem} Let $G$ be an absolutely simple, simply connected algebraic group which is isotropic over $k$, a local 
field. Let $S$ be a maximal split torus in $G$, and $\alpha$ a root of $G$ with respect to $S$. Let $G_\alpha$ be the 
(simply connected semi-simple) group generated by the root subgroups $U_\alpha$ and $U_{-\alpha}$. Then if $\alpha$ is a long root
in the relative root system of $G$ with respect to $S$ (which if the root system is not reduced means that $\alpha/2$ is a root), 
the restriction map from   
$H^2(G(k), \Q/\Z) $ to $H^2(G_{\alpha}(k), \Q/\Z) $ is injective.
\end{theorem}

Regarding question 1 above, let $E/F$ be a finite separable extension of non-Archimedean local fields. Let $m$ be an integer such
that the $m$-th roots of unity are contained in $F^\times$. Let $(-,-)_m^F$ denote 
the Hilbert symbol on $F$ with values in the $m$-th roots of unity in $F$, and similarly, let $(-,-)_m^E$ denote 
the Hilbert symbol on $E$ with values in the $m$-th roots of unity in $E$. 
Then it is known that,
$$(a,b)_m^E = (a,b)_m^{F\, d},$$
where $a,b \in F^\times$, and $d$ is the degree of the field extension $E/F$. By the way central extensions of $G(E)$ and $G(F)$ are constructed using Hilbert symbols,
it follows that 
the restriction from $H^2(G(E),\Q/\Z)$ to $H^2(G(F),\Q/\Z)$ lands inside $d \cdot H^2(G(F),\Q/\Z)$, in particular any
central extension of $G(E)$ by $\mu_d \subset F^\times$ becomes trivial when restricted to $G(F)$ whenever the degree of $E$ over $F$ is a multiple of 
$d$.

Regarding question 2 above, there is a general recipe due to Deligne. For this assume that $G$ and $H$ are simple, 
simply connected, split groups with maximal tori $T$ and $S$, and a morphism $\phi: H \rightarrow G$ defined over $F$ taking $S$ to $T$.
Assume that $\alpha^\vee: \G_m\rightarrow S \subset H$ is a coroot in $H$ corresponding to a long root $\alpha$ for $H$ 
 which under $\phi$ goes to the coroot $\phi(\alpha^\vee)$ for $G$.   Fix a Weyl group invariant positive definite integral bilinear
form on the cocharacter group of $T$ such that the corresponding quadratic form $Q_G$ takes the value 1 on any coroot corresponding to a long 
root of $T$ in $G$. Let $d = Q_G(\phi(\alpha^\vee))$. 
Then the restriction from $H^2(G(F),\Q/\Z)$ to $H^2(H(F),\Q/\Z)$ has its image equal to  $d \cdot H^2(H(F),\Q/\Z)$.

\begin{remark} Although it is not meaningful to talk about central extensions of $G(\bar{k})$ for the algebraic closure $\bar{k}$ of a 
non-Archimedean local field $k$, since there are none which are non-trivial, for each finite Galois extension $K$ of $k$, there is a rich supply of central extensions, 
say $E(K)$ of $G(K)$ which carry the Galois action too; the only (slight) issue is that the Gal$(K/k)$-invariants in $E(K)$ is an extension 
of $G(k)$ is not necessarily a non-trivial extension---still a perfectly fine context to think about basechange issues.
\end{remark}

\section{Some results about Galois groups of number fields}

The notion of Schur multiplier, and the dual notion of constructing extensions of a group are also studied for Galois groups of number fields and local fields, and have important implications in the subject. We simply state two most important results on these.

\begin{theorem}[Tate]
For any number field $F$, $H^{2}(Gal(\bar{F}/F), \C^{\times}) = \{ 1 \}$.
\end{theorem}
\begin{theorem}[Neukirch]
The group of outer automorphism $\Out(Gal(\bar{\Q}/\Q)) = \{ 1 \}$. 
\end{theorem}


\begin{thebibliography}{99}
\bibitem{Brown82}
Kenneth S. Brown:
\emph{Cohomology of groups}.
Graduate Texts in Mathematics, 87. Springer-Verlag, New York-Berlin, 1982. x+306 pp.

\bibitem{Deo} Vinay, Deodhar. \emph{On central extensions of rational points of algebraic groups.}
 Amer. J. Math. 100 (1978), no. 2, 303-386.

\bibitem{Mil71}
  John Milnor:
 \emph{Introduction to algebraic K-theory}. 
 Annals of Mathematics Studies, {\bf No. 72}, Princeton University Press, Princeton, N. J.; 1971.

\bibitem{Matsumoto}
Hideya Matsumoto:
\emph{Sur les sous-groupes arithm\'etiques des groupes semi-simples d\'eploy\'es}
Ann. Sci. \'Ecole Norm. Sup. (4) 2 (1969) 1-62. 

\bibitem{PrRa84}
Gopal Prasad and M. S. Raghunathan:
\emph{Topological central extensions of semisimple groups over local fields}. 
Ann. of Math. (2) {\bf 119} (1984), no. 1, 143-201.


\bibitem{PrRa} Gopal Prasad and Andrei Rapinchuk: \emph{Computation of the metaplectic kernel.} Inst. Hautes \'Etudes Sci. Publ. Math. No. 84 (1996), 
91-187 (1997)

\bibitem{Steinberg62}
R. Steinberg:
\emph{G{\'e}ne{\'e}rateurs, relations et rev{\^e}tements de groupe algebriques}.
Colloq. Th{\'e}orie des groupes algebriques, Bruxelles (1962) 113-127.

\bibitem{Steinberg62}
R. Steinberg:
\emph{Lecture on Chevalley groups}.
Yale University (1967).

\bibitem{Wikipedia}
Wikipedia:
\emph{List of finite simple groups}. Available at: \\
{http://en.wikipedia.org/wiki/List\_of\_finite\_simple\_groups}.
\end{thebibliography}
\end{document}